\documentclass[12pt]{article}
\usepackage{amssymb}
\usepackage{amsthm}
\usepackage{amsmath}
\usepackage{graphicx}
 \newtheorem{thm}{Theorem}[section]
 
 \newtheorem{lem}[thm]{Lemma}

 \newtheorem{rem}[thm]{Remark}
 \numberwithin{equation}{section}

\newcommand{\Bx}{r}

\newcommand{\al}{\alpha}

\newcommand{\mathd}{\mathrm{d}}

\begin{document}
\title{On Finite Time Singularity and Global Regularity of an Axisymmetric Model for the 3D Euler Equations}
\author{Thomas Y. Hou\footnote{Computing \& Mathematical Sciences, California Institute of
 Technology, Pasadena, CA 91125, USA}\and Zhen Lei\footnote{School of Mathematical Sciences; LMNS and Shanghai Key
  Laboratory for Contemporary Applied Mathematics, Fudan
  University, Shanghai 200433, P. R. China.}
\and  Shu Wang\footnote{College of Applied Sciences, Beijing
University of Technology, Beijing 100124, China. }\and Chen
Zou\footnote{Department of Mechanics and Aerospace Engineering, COE,
Peking University, Beijing 100871, China }}
\date{\today}
\maketitle

\begin{abstract}
We investigate the large time behavior of an axisymmetric  model for the
3D Euler equations. In \cite{HL09}, Hou and Lei proposed
a 3D model for the axisymmetric incompressible Euler and
Navier-Stokes equations with swirl. This model shares many
properties of the 3D incompressible Euler and Navier-Stokes equations.
The main difference between the 3D model of Hou and Lei and the
reformulated 3D Euler and Navier-Stokes equations is that the
convection term is neglected in the 3D model. In \cite{HSW09}, the
authors proved that the 3D inviscid model can develop a finite
time singularity starting from smooth initial data on a rectangular
domain. A global well-posedness result was also proved for a
class of smooth initial data under some smallness condition.
The analysis in \cite{HSW09} does not apply to the case when
the domain is axisymmetric and unbounded in the radial direction.
In this paper, we prove that the 3D inviscid model with
an appropriate Neumann-Robin boundary condition will develop
a finite time singularity starting from smooth initial data
in an axisymmetric domain. Moreover, we prove that the 3D inviscid
model has globally smooth solutions for a class of large smooth
initial data with some appropriate boundary condition.
\end{abstract}

\maketitle





\section{Introduction}

Whether the 3D incompressible Navier-Stokes equations can develop
a finite time singularity from smooth initial data with finite
energy is one of the most challenging questions in nonlinear
partial differential equations \cite{Fefferman}. There have been
many previous studies devoted to this challenging question,
see e.g. \cite{Constantin86,CS04,CFM96,DHY05,DHY06,Chae07} and
two recent review articles \cite{BT07,Constantin07}.
There have been a number of attempts to investigate possible
finite time singularity formation of the 3D Euler equations 
numerically, see e.g. \cite{GS91,Kerr93,BP94}. So far, the 
numerical evidence for a finite time blow-up is not yet conclusive 
\cite{HL06,HL07,Hou09}.

Global regularity
results for the 3D Navier-Stokes equations have been obtained
using energy estimates under some smallness assumption on the
initial data \cite{Lady70,MB02, LL11}. It is well known that the
convection term does not play an essential role in energy
estimates due to the incompressibility of the velocity field.
Thus, it is not clear how convection may contribute to global
well-posedness of the 3D incompressible Navier-Stokes equations.

In \cite{HL08}, Hou and Li investigated the stabilizing effect of 
convection via an exact 1D model. They found that the convection 
term plays an essential role in canceling the destablizing vortex 
stretching terms in this 1D model. This observation enabled them
to obtain a pointwise estimate via a Liapunov function which
controls the dynamic growth of the derivative of vorticity.
Motivated by the work of \cite{HL08}, Hou and Lei further 
investigated the role of convection by constructing the following 
3D model of the axisymmetric Navier-Stokes equations with swirl
\cite{HL09}: 
\begin{equation}\label{3DInviscidModel}
\begin{cases}
\partial_tu = \nu\big(\partial_r^2 + \frac{3}{r}\partial_r
  + \partial_z^2\big)u + 2 u\partial_z\psi,\\[-4mm]\\
\partial_t\omega = \nu\big(\partial_r^2 + \frac{3}{r}\partial_r
  + \partial_z^2\big)\omega + \partial_z(u^2),\\[-4mm]\\
- \big(\partial_r^2 + \frac{3}{r}\partial_r
  + \partial_z^2\big)\psi = \omega.
\end{cases}
\end{equation}
When $\nu=0$, we refer to the above model as the 3D inviscid model.
This model derived by using a reformulated Navier-Stokes
equations in terms of a set of new variables
$(u,\omega,\psi)=(u^\theta,\omega^\theta,\psi^\theta)/r$, where
$r=\sqrt{x_1^2+x_2^2}$, $u^\theta$ is the angular velocity,
$\omega^\theta$ the angular vorticity, and $\psi^\theta$ the
angular stream function, see \cite{HL08,HL09}. The only difference
between this 3D model and the reformulated Navier-Stokes equations
is that we neglect the convection term in the model. This new 3D
model shares several well known properties of the full 3D Euler or
Navier-Stokes equations \cite{HL09}. These include an energy
identity for smooth solutions, an artificial incompressible
constraint, a non-blowup criterion of Beale-Kato-Majda type
\cite{BKM84}, a non-blowup criterion of Prodi-Serrin type
\cite{Prodi,Serrin}, and a partial regularity result for the model
\cite{HL09b}, which is an analogue of the
Caffarelli-Kohn-Nirenberg theory \cite{CKN82,Lin88} for the full
Navier-Stokes equations.

Despite the striking similarity at the theoretical level, this 3D
model seems to have a very different behavior from that of the Euler
or Navier-Stokes equations. Numerical study in \cite{HL09} seems to
suggest that the model develop a potential finite time singularity
from smooth initial data with finite energy. However, the mechanism
that leads to the singular behavior of the 3D model seems to be
destroyed when the convection term is added back to the model.
In a recent paper \cite{HSW09}, Hou, Shi and Wang proved rigorously
that the inviscid model can indeed develop a finite time singularity
for a class of smooth initial data with some Neumann-Robin boundary
condition. Moreover, they proved a global well-posedness result for
a class of small smooth initial data. The analysis in \cite{HSW09}
was carried out for a rectangular domain, which does not
apply to the axisymmetric domain considered in this paper.

In this paper, we prove that the 3D inviscid model with some
appropriate boundary conditions of Neumann-Robin type can develop
a finite time singularity starting from smooth initial data on a
bounded or an unbounded exterior axisymmetric domain. Moreover, we
obtain a global well-posedness result for a class of large smooth
initial data under some Dirichlet boundary condition.

Our analysis in this paper is similar in spirit to that of
\cite{HSW09}. However, there are several new ingredients in this
work. The first one is that our blowup result applies to a bounded
or unbounded exterior domain in the axisymmetric geometry. The
local well-posedness of the 3D model with a boundary condition of
the Neumann-Robin type is more complicated than the case
considered in \cite{HSW09} and involves the use of the modified
Bessel function. The second ingredient is to construct a special
positive test function that satisfies several requirements. In the
case of a bounded domain, we use the first eigenfunction of the
Laplacian operator with homogeneous Dirichlet boundary condition.
In the case of an unbounded exterior domain, a positive and
decaying eigenfunction of the Laplacian operator does not exist.
This requires us to relax the constraints on the test function.
The third ingredient is an improved estimate which enables us to
establish the global regularity of the 3D inviscid model for a
class of initial boundary value problem with large initial data.

The paper is organized as follows: In Section 2 we will state our
main results and present our main ideas of their proofs.
Section 3 is devoted to proving the finite time singularity of the
3D inviscid model.  In Section 4 we present our global well-posedness
result. Finally, we prove the local well-posedness of the initial
boundary value problem in Section 5.

\section{Main Results}

In this section we set up our problem and state our main results.
First of all, let us recall that the 3D model
\eqref{3DInviscidModel} is formulated in terms of a set of new
variables 
$u = \frac{u^\theta}{r}$, $\omega =\frac{\omega^\theta}{r}$,
$\psi = \frac{\psi^\theta}{r}.$
Due to this change of variables, the original three-dimensional
Laplacian operator for $\psi^\theta$ has been changed into a
five-dimensional Laplacian operator for $\psi$ in the axisymmetric
cylindrical coordinate in \eqref{3DInviscidModel}. Thus we can
reformulate our 3D inviscid model in the three-dimensional
axisymmetric cylindrical domain
\begin{equation}\nonumber
D(\gamma_1, \gamma_2) = \{(x_1, x_2, z) \in \mathbb{R}^3: \gamma_1
\leq r = \sqrt{x_1^2 + x_2^2} < \gamma_2; 0 \leq z \leq 1\}
\end{equation}
as a five-dimensional problem in the
axisymmetric cylindrical domain
\begin{equation}\nonumber
\Omega(\gamma_1, \gamma_2) = \{(x_1, \cdots, x_4, z) \in
\mathbb{R}^5: \gamma_1 \leq r = \big(\sum_{j =
1}^4x_j^2\big)^{\frac{1}{2}} < \gamma_2; 0 \leq z \leq 1\} \;.
\end{equation}
It is much more convenient to perform the well-posedness and the
finite time blow-up analysis for our 3D model in the
five-dimensional setting. In the remaining part of the paper, we
will carry out our analysis in this axisymmetric five-dimensional
domain.

Denote
\begin{align}\label{set-1}
S_{\rm exterior}\equiv\left\{\beta > 0 \quad | \quad \beta\neq\frac{
k\int_0^{\infty}e^{-k\cosh(\theta)}\cosh^2(\theta)d\theta}{
\int_0^{\infty}e^{-k\cosh(\theta)}\cosh(\theta)d\theta},\ \forall
k\in\mathbb{Z}^{+}\right\}.
\end{align}
Now we are ready to state the first result of this paper which is
concerned with the local well-posedness of classical solutions to
the 3D inviscid model with a Neumann-Robin
type boundary condition on the exterior domain $\Omega(1, \infty)$.
\begin{thm}\label{thm-local-exterior}
Let $S_{\rm exterior}$ be defined in \eqref{set-1} and $\beta \in
S_{\rm exterior}$. Consider the initial-boundary value problem of
the 3D model \eqref{3DInviscidModel} with the initial data
\begin{equation}\label{data-1}
u_0 \in H^3(\Omega(1, \infty)),\ \ \psi_0 \in H^4(\Omega(1,
\infty)),
\end{equation}
\begin{equation}\label{data-11}
u_0^2 > 0\ {\rm for}\ z \neq 0\ {\rm and}\ z \neq 1,\quad u_0|_{z
= 0} = u_0|_{z = 1} = 0
\end{equation}
and the Neumann-Robin type boundary condition
\begin{equation}\label{bd-1}
(\psi_r + \beta\psi)|_{r = 1} = 0,\quad {\psi_z}|_{z = 0} =
{\psi_z}|_{z = 1} = 0.
\end{equation}
Then there exists a unique smooth solution $(u, \psi)$ to the 3D
inviscid model with the initial data
\eqref{data-1}-\eqref{data-11} and the boundary condition
\eqref{bd-1} on $[0, T)$ for some $T
> 0$. Moreover, we have
\begin{equation}\label{2-4-1}
u \in C\big([0, T), H^3(\Omega(1, \infty))\big),\quad \psi \in
C\big([0, T), H^4(\Omega(1, \infty))\big).
\end{equation}
\end{thm}

The proof of Theorem \ref{thm-local-exterior} relies on an
important property of the elliptic operator with the mixed
Neumann-Robin type boundary condition. Note that the boundary
condition in \eqref{bd-1} is not the third type. To recover the
standard elliptic regularity estimate, we need to study the
spectral property of the differential operator which requires that
$\beta \in S_{\rm exterior}$. We remark that the local
well-posedness analysis for \eqref{3DInviscidModel} with a mixed
Dirichlet-Robin boundary condition has been established in
\cite{HSW09} where the non-standard boundary condition is imposed
on the axial direction. Here our non-standard boundary condition
is imposed on the radial direction. The local well-posedness in
such case involves the use of modified Bessel function and is more
involved. The proof of Theorem \ref{thm-local-exterior} will be
deferred to Section 5.

Next, we state the finite time blowup result of the exterior
problem \eqref{3DInviscidModel}, \eqref{data-1},\eqref{data-11}
and \eqref{bd-1}.
\begin{thm} \label{thm-bu-exterior}
Suppose that all the assumptions in Theorem
\ref{thm-local-exterior} are satisfied.
Let $\beta\geq2+2\sqrt{1+\frac{\pi^2}{4}}$,
$\alpha=\frac{\beta}{2}$ and
\begin{align}\label{2-4}
\phi(r,z)=\theta(r)\sin(\pi z),\ \theta(r)=e^{-\alpha r^2},\
\Phi(r,z)=\phi_{rr}+\frac{3}{r}\phi_r+\phi_{zz}.
\end{align}
Then the solution of the 3D inviscid model
will develop a finite time singularity in the $H^3$ norm provided
that
\begin{align}\label{2-5}
\int_0^1 \int_1^\infty \Phi\log (u^2_0) r^3 drdz > 0,\quad
\int_0^1 \int_1^\infty \Phi\psi_{0z} r^3 drdz > 0
\end{align}
and
\begin{eqnarray}\label{2-6}
\int_0^1 \int_1^\infty\partial_z\psi_0 \Phi r^3drdz\Big)^2
\geq \frac{16}{c_0}\Big(\int_0^1 \int_1^\infty(\log u_0^2)\Phi
r^3drdz\Big)^3,
\end{eqnarray}
where $c_0$ is a positive constant defined in \eqref{c0}.
\end{thm}
The proof of Theorem \ref{thm-bu-exterior} is in spirit similar to
that of Theorem 1.1 in \cite{HSW09}. The main ingredient here is
to construct a special positive test function which meets our
requirements on the unbounded domain case. In
\cite{HSW09}, the authors can take the product of sine
functions in the $x-y$ domain as a testing function since the
domain is a rectangular domain. In the case of a unbounded exterior
domain in the axisymmetric geometry, a positive and decaying
eigenfunction does not seem to exist. This requires us to relax
the constraints on the test function. We will present the details
of the proof of Theorem \ref{thm-bu-exterior} in Section 3.

\begin{rem}
We remark that similar results in Theorem \ref{thm-local-exterior}
and \ref{thm-bu-exterior} are also true in the bounded domain case
$\Omega(0, 1)$. We will not present this result here. A more interesting
result would be to obtain a similar result for the case of
$\Omega(0, \infty)$ with an initial boundary condition whose energy is
conserved in time. This would require a different technique.
We will report it in a forthcoming paper.
\end{rem}

In Theorem 1.1 of \cite{HSW09}, the authors proved the finite time
singularity of the 3D inviscid model in a
rectangular domain. Their Robin boundary condition is imposed on
the axial boundary $z = 1$.  Our next theorem extends their result
to the case of a bounded axisymmetric domain $\Omega(0, 1)$.

Let $(\lambda_k, \theta_k(\Bx))$ be the eigenvalue-eigenfunction
pair of the following Dirichlet eigenvalue problem:
\begin{align}\label{eg1}
\begin{cases}
-(\partial_r^2 + \frac{3}{r}\partial_r)  \theta_k=\lambda_k \theta_k,\quad r\in[0,1),\\
\theta_k|_{r=1}=0.
\end{cases}
\end{align}
Define
\begin{eqnarray}\label{2-8}
S_{\rm interior} \equiv \{\beta > 0 \;|\; \beta\ne \lambda_k,
\beta \neq
\frac{\sqrt{\lambda_k}(e^{\sqrt{\lambda_k}}+e^{-\sqrt{\lambda_k}})}{e^{\sqrt{\lambda_k}}-e^{-\sqrt{\lambda_k}}}
\mbox{ for all } \; k=1, 2, \cdots\},
\end{eqnarray}
and
\begin{align}\label{phi-tf}
 \phi(\Bx,z)=\frac{e^{-\alpha (z-1)}+e^{\alpha
(z-1)}}{2} \theta_1(\Bx).
\end{align}
Moreover, we assume that $\alpha$, $\beta$ satisfy
\begin{equation}\label{2-10}
0<\al<\sqrt{\lambda_1},\quad \beta = \frac{\lambda_1}{\al}
\frac{e^{\al }-e^{-\al }}{ e^{\al }+e^{-\al }}
>{\sqrt{\lambda_1}}\left(\frac{e^{\sqrt{\lambda_1}}+e^{-\sqrt{\lambda_1}}}
{e^{\sqrt{\lambda_1}}-e^{-\sqrt{\lambda_1}}} \right ).
\end{equation}
We can prove the following local-well-posedness and finite time blow-up
results.
\begin{thm}\label{thm-bd}
(A) {\it Local well-posedness}.
Let $S_{\rm interior}$ be defined in \eqref{2-8} and $\beta \in
S_{\rm interior}$. Consider the initial-boundary value problem of
the 3D model \eqref{3DInviscidModel} with the initial data
\begin{equation}\label{data-2}
u_0 \in H^3(\Omega(0, 1)),\ \ \psi_0 \in H^4(\Omega(0, 1)),
\quad u_0|_{z = 0} = u_0|_{z = 1} = 0
\end{equation}
and the Dirichlet-Robin type boundary condition
\begin{equation}\label{bd-2}
\psi|_{r = 1} = \psi|_{z = 1} = 0,\quad {\psi_z + \beta\psi}|_{z =
0} = 0.
\end{equation}
Then there exists a unique smooth solution $(u, \psi)$ to the 3D
inviscid model with the initial data
\eqref{data-2} and the boundary condition \eqref{bd-2} on $[0, T)$
for some $T > 0$. Moreover, we have
\begin{equation}\label{2-13}
u \in C\big([0, T), H^3(\Omega(0,1))\big),\quad \psi \in
C\big([0, T), H^4(\Omega(0,1))\big).
\end{equation}
(B) {\it Finite time blow-up}.
Suppose \eqref{2-10} is satisfied and $u_0^2 > 0 $ for $0 < z < 1$.
If $u_0$ and $\psi_0$ satisfy
\begin{align}\nonumber
\int_0^1 \int_0^1 (\log u_0^2)\phi r^3\mathd {\Bx}\mathd z > 0,\quad
\int_0^1 \int_0^1 \psi_{0z}\phi r^3\mathd {\Bx}\mathd z
> 0,
\end{align}
and
\begin{eqnarray}\nonumber
\Big(\int_0^1 \int_0^1\partial_z\psi_0 \phi r^3drdz\Big)^2 \geq
\frac{1}{c_1}\Big(\int_0^1 \int_0^1 (\log u_0^2)\phi
r^3drdz\Big)^3,
\end{eqnarray}
for some absolute constant $c_1 > 0$, then the
solution of the 3D inviscid model will
develop a finite time singularity in the $H^3$ norm.
\end{thm}

The local well-posedness of the initial-boundary value problem for
3D model in Theorem \ref{thm-bd} can be carried out using almost
the same argument as that in \cite{HSW09} by replacing the first
eigenfunction $\sin \pi x_1\sin\pi  x_2\sin\pi  x_3\sin\pi  x_4$
of the rectangular domain by the first eigenfunction $\phi_1(x)$ of the
unit ball in $\mathbb{R}^4$. The finite time blow-up result can
be proved in exactly the same way as in \cite{HSW09}.
Since the analysis is essentially the same as that in \cite{HSW09},
we will omit the proof of Theorem \ref{thm-bd} in this paper.

The last theorem extends the global existence result in Theorem
6.1 in \cite{HSW09} to the axisymmetric domain $\Omega(0, 1)$.
The most interesting aspect of this result is that the initial
condition can be made as large as we wish in the Sobolev space.
This is achieved by effectively imposing a sufficiently large
negative boundary
condition for $\partial_z\psi |_{\partial \Omega(0, 1)} = -M$
with $M$ being a large positive constant.

\begin{thm}\label{thm1.2}
Let $M>0 $ an arbitrarily large constant and $s \geq 3$.
Assume that $u_0 \in H^s(\Omega(0, 1))$,  $\psi_0 \in
H^{s+1}(\Omega(0, 1))$ with $u_0(r, 0) = u_0(r, 1) = 0$ and
$\psi_0|_{r = 1} = - Mz$, $\partial_z\psi_0|_{z = 0} =
\partial_z\psi_0|_{z = 1} = - M$. Then the 3D inviscid model
with the initial and boundary data
\begin{equation}\label{data-3}
u(0, \cdot) = u_0(\cdot),\quad \psi(0, \cdot) = \psi_0(\cdot),
\end{equation}
and
\begin{equation}\label{data-31}
\psi(t, 1, z) = - Mz,\ \partial_z\psi(t, r, 0) = \partial_z\psi(t,
r, 1) = -M
\end{equation}
admits a unique global smooth solution $(u, \psi)$ with $u \in
C\big([0, \infty); H^s(\Omega(0, 1))\big)$, $\psi \in C\big([0,
\infty); H^{s+1}(\Omega(0, 1))\big)$ provided that
\begin{eqnarray}\label{data-4}
\|\nabla \partial_z\psi_0\|_{H^{s-1}} \leq M/(8C_s^2), \quad
\|u_0^2\|_{H^s} \leq M^2/(4C_s^3),
\end{eqnarray}
where $C_s$ is an absolute positive constant depending on $s$ and
$\Omega$ only. Moreover, we have
\begin{equation}\label{ap-3}
\|u(t)^2\|_{H^s(\Omega(0, 1))}\le
C_s\|u_0^2\|_{H^s(\Omega(0, 1))}e^{-2Mt},
\end{equation}
and
\begin{equation}\label{ap-4}
\|\nabla\partial_z\psi\|_{H^{s-1}(\Omega(0, 1))}\le
\|\nabla\partial_z\psi_0\|_{H^{s-1}(\Omega(0, 1))} +
\frac{C_s}{2M}\|u_0^2\|_{H^{s}(\Omega(0, 1))}.
\end{equation}
\end{thm}

We will present the proof of Theorem \ref{thm1.2} in Section 4.

\section{Finite Time Singularity in the Exterior Domain}

In this section, we present the proof of Theorem
\ref{thm-bu-exterior}, which shows that the 3D inviscid model develops
a finite time singularity in the exterior domain.

\noindent
{\bf Proof of Theorem \ref{thm-bu-exterior}}.
We will prove Theorem \ref{thm-bu-exterior} by
contradiction, as in \cite{HSW09}. By the local well-posedness result,
we know that the initial boundary value problem of the 3D inviscid model
has a unique solution $u \in H^3(\Omega(1, \infty))$ and
$\psi \in H^4(\Omega(1, \infty))$ for $ 0 \leq t \leq T$ for some $T>0$.
Let $T_b$ be the largest time for which our 3D inviscid
model has a unique smooth solution $u \in H^3(\Omega)$ and
$\psi\in H^4(\Omega)$ for $ 0 \leq t < T_b$. We will prove that
$T_b < \infty$. Suppose that $T_b = \infty$. Then we have
\begin{equation}\label{3-1}
u \in C\big([0, \infty), H^3(\Omega(1, \infty))\big),\quad \psi
\in C\big([0, \infty), H^4(\Omega(1, \infty))\big).
\end{equation}
We will prove that \eqref{3-1} would lead to a contradiction.

Define $\phi$ and $\Phi$ as in Theorem \ref{thm-bu-exterior}.
By a straightforward calculation, we have
\begin{align}\label{3-2}
\Phi(r,z)=[4\alpha^2r^2-(8\alpha+\pi^2)]\phi(r, z).
\end{align}
For $\alpha\geq1+\sqrt{1+\frac{\pi^2}{4}}$, it is easy to get that
\begin{equation}\nonumber
4\alpha^2r^2-(8\alpha+\pi^2) \geq 4\alpha^2-(8\alpha+\pi^2) \geq
0\quad {\rm for}\ r \geq 1.
\end{equation}
Consequently, we have
\begin{align}\label{3-3}
\Phi(r,z) \geq 0\quad {\rm for}\  1 \leq r, 0 \leq z \leq 1.
\end{align}
On the other hand, due to the boundary condition on the initial
data $u_0$ in \eqref{data-11}, one can use the first equation in
\eqref{3DInviscidModel} with $\nu=0$ to solve $u$:
\begin{equation}\nonumber
u^2(t, r, z) = u_0^2(r, z)\exp\{4\int_0^t\partial_z\psi(s, r, z)
ds\}\quad {\rm for}\ z \neq 0, 1.
\end{equation}
Hence, by continuity, as long as the solution is smooth, one also
has
\begin{equation}\label{3-4}
u(t, r, 0) = u(t, r, 1) = 0.
\end{equation}

Multiplying the second equation in \eqref{3DInviscidModel} with
$\nu=0$  by
$\phi_z$ and integrating over $\Omega(1, \infty)$, we have

\begin{align}\label{3-4-a}
- \int_0^1 \int_1^\infty (\partial_r^2 + \frac{3}{r}\partial_r +
\partial_z^2)\psi_t \phi_z r^3drdz = \int_0^1 \int_0^\infty
\partial_zu^2\phi_zr^3drdz.
\end{align}
On one hand, using \eqref{3-4} and performing integration by parts,
we get
\begin{align}\label{3-4-b}
\int_0^1\int_1^\infty \partial_zu^2\phi_zr^3drdz = -
\int_0^1\int_1^\infty u^2\partial_z^2\phi r^3drdz = \pi^2
\int_0^1 \int_1^\infty u^2\phi r^3drdz.
\end{align}
On the other hand, using the boundary condition \eqref{bd-1}
and performing integration by parts, we obtain by
noting that $\theta_r(1)+\beta\theta(1) = 0$ that
\begin{eqnarray}\label{3-4-c}
&&- \int_0^1 \int_1^\infty (\partial_r^2 + \frac{3}{r}\partial_r
  + \partial_z^2)\psi_t \phi_z r^3drdz\\\nonumber
&&= \frac{d}{dt}\int_0^1(\psi_r\phi_z - \psi\partial_{rz}^2\phi)|_{r=1}dz
  - \frac{d}{dt}\int_0^1 \int_1^\infty \psi(\partial_r^2 +
  \frac{3}{r}\partial_r) \phi_z r^3drdz\\\nonumber
&&\quad +\ \frac{d}{dt}\int_0^1 \int_1^\infty \psi_z\partial_z^2
  \phi r^3drdz\\\nonumber
&&= \pi\theta(1)\frac{d}{dt}\int_0^1(\cos\pi z)\big(\psi_r
  + \beta \psi \big)|_{r=1} dz
  + \frac{d}{dt}\int_0^1 \int_1^\infty \psi_z\Phi r^3drdz\\
&&= \frac{d}{dt}\int_0^1 \int_1^\infty  \psi_z\Phi r^3drdz.
\end{eqnarray}
Combining \eqref{3-4-c} with \eqref{3-4-a}-\eqref{3-4-b} gives
\begin{eqnarray}\label{3-5}
\frac{d}{dt}\int_0^1 \int_1^\infty \psi_z\Phi r^3drdz = \pi^2
\int_0^1 \int_1^\infty u^2\phi r^3drdz.
\end{eqnarray}

Multiplying the first equation in \eqref{3DInviscidModel} 
with $\nu=0$ by $\Phi$ and integrating on $\Omega(1, \infty)$ yield
\begin{align}\label{3-6}
\frac{d}{dt}\int_0^1 \int_1^\infty (\log u^2) \Phi r^3drdz =
4\int_0^1 \int_1^\infty  \psi_{z} \Phi r^3drdz.
\end{align}
Since $\int_0^1 \int_1^\infty \partial_z\psi_0\Phi r^3drdz > 0$,
we conclude from \eqref{3-5} that
\begin{eqnarray}\label{3-7}
\int_0^1 \int_1^\infty \psi_z\Phi r^3drdz > 0\quad {\rm for\
all}\ t \geq 0.
\end{eqnarray}
It follows form \eqref{3-6}-\eqref{3-7} and the condition
$\int_0^1 \int_1^\infty (\log u_0^2)\Phi r^3drdz > 0$ that
\begin{align}\label{3-8}
\int_0^1 \int_1^\infty (\log u^2)\Phi r^3drdz > 0\quad {\rm for\
all}\ t \geq 0.
\end{align}
Consequently, we obtain by using \eqref{3-2} and \eqref{3-8} that
\begin{eqnarray}\label{3-9}
&&\Big(\int_0^1 \int_1^\infty (\log u^2)\Phi
  r^3drdz\Big)^2 \leq \Big(\int_0^1 \int_1^\infty (\log^+ u^2)\Phi
  r^3drdz\Big)^2\\\nonumber
&&\leq 4\Big(\int_0^1 \int_1^\infty  |u|\Phi
  r^3drdz\Big)^2\\\nonumber
&&\leq 4 \int_0^1 \int_1^\infty |u|^2\phi
  r^3drdz \int_0^1 \int_1^\infty \big(\frac{\Phi}{\phi}\big)^2\phi
  r^3drdz\\\nonumber
&&\leq \frac{8c_0\pi^2}{3}\int_0^1 \int_1^\infty |u|^2\phi
  r^3drdz,
\end{eqnarray}
where $c_0$ is defined by
\begin{eqnarray}\label{c0}
c_0 = \frac{3}{2\pi^2}\int_0^1 \int_1^\infty \big(4\alpha^2r^2-(8\alpha+\pi^2)\big)^2\phi
  r^3drdz < \infty.
\end{eqnarray}
Using \eqref{3-5}, \eqref{3-6} and \eqref{3-9}, we get
\begin{eqnarray} \label{3-10}
\frac{d^2}{dt^2}\int_0^1 \int_1^\infty (\log u^2)\Phi r^3drdz &= &
  4\pi^2 \int_0^1 \int_1^\infty  u^2\phi r^3drdz\\\nonumber
&\geq & \frac{3}{2c_0}\Big(\int_0^1 \int_1^\infty (\log u^2) \Phi
  r^3drdz\Big)^2.
\end{eqnarray}
Let
\begin{equation}
\label{def-Y}
Y(t) \equiv \int_0^1 \int_1^\infty (\log u^2) \Phi r^3drdz\;.
\end{equation}
Then \eqref{3-10} implies
\begin{equation}
\label{3-11}
Y''(t) \geq \frac{3}{2c_0} Y(t).
\end{equation}
It follows from \eqref{3-6}-\eqref{3-7} that $Y'(t) > 0$.
Multiplying \eqref{3-11} by $Y'(t)$ and integrating in
time from 0 to $t$, we get
\begin{eqnarray} \nonumber
&&\Big(Y'(t)\Big)^2 \geq \Big(Y'(0)\Big)^2 - \frac{1}{c_0} Y(0)^3
+\ \frac{1}{c_0} Y(t)^3\\\nonumber
&&= \Big(4\int_0^1 \int_1^\infty \partial_z\psi_0 \Phi
  r^3drdz\Big)^2 - \frac{1}{c_0}\Big(\int_0^1 \int_1^\infty 
  (\log u_0^2)\Phi r^3drdz\Big)^3
+\ \frac{1}{c_0} Y(t)^3 \;.
\end{eqnarray}
The condition \eqref{2-6} implies
\begin{eqnarray} \label{dY(t)}
\Big(Y'(t)\Big)^2 \geq \frac{1}{c_0} Y(t)^3.
\end{eqnarray}
Since $ Y'(t) >0$, it is easy to solve \eqref{dY(t)} to
conclude that
\begin{eqnarray}\label{4-5}
\int_0^1 \int_1^\infty (\log u^2)\Phi r^3drdz \geq
\frac{4c_0\int_0^1 \int_1^\infty (\log u_0^2) \Phi
r^3drdz}{\Big(2\sqrt{c_0} - t\sqrt{\int_0^1 \int_1^\infty 
(\log u_0^2) \Phi r^3drdz}\Big)^2} \;\;.
\end{eqnarray}
Note that \eqref{3-2} gives
\begin{equation}
\Phi(r,z)=[4\alpha^2r^2-(8\alpha+\pi^2)]\phi(r, z)
= [4\alpha^2r^2-(8\alpha+\pi^2)] e^{-\alpha r^2} \sin(\pi z).
\end{equation}
Since $r\geq 1$, it is easy to show that $\Phi$ satisfies
\begin{equation}
0 \leq \Phi(r,z) \leq 4\alpha^2r^2 e^{-\alpha r^2} \leq
 4\alpha^2 e^{-\alpha}, \quad r \geq 1.
\end{equation}
Thus, we have
\begin{eqnarray}
\label{4-6}
\int_0^1 \int_1^\infty (\log u^2)\Phi r^3drdz &\leq &
\int_0^1 \int_1^\infty (\log^+ u^2)\Phi r^3drdz \\ \nonumber
&\leq & 4\alpha^2 e^{-\alpha} \int_0^1 \int_1^\infty u^2 r^3drdz .
\end{eqnarray}
Combining \eqref{4-5} with \eqref{4-6} gives
\begin{equation}
\label{4-7}
4\alpha^2 e^{-\alpha} \int_0^1 \int_1^\infty  u^2 r^3drdz \geq
\frac{4c_0\int_0^1 \int_1^\infty (\log u_0^2)\Phi
r^3drdz}{\Big(2\sqrt{c_0} - t\sqrt{\int_0^1 \int_1^\infty 
(\log u_0^2)\Phi r^3drdz}\Big)^2} \; .
\end{equation}
On one hand,  we have $\int_0^1 \int_1^\infty u^2 r^3drdz < \infty$
for all times by \eqref{3-1}. On the other hand, the right hand side
of \eqref{4-7} will blow up as
$t \rightarrow \frac{2\sqrt{c_0}}{\sqrt{\int_0^1\int_0^\infty(\log u_0^2)\Phi r^3drdz}}$. This is clearly a contradiction. This contradiction implies
that the assumption \eqref{3-1} can not be valid and the solution
must blow up in a finite time in the $H^3$ norm no later than
$T^* = \frac{2\sqrt{c_0}}{\sqrt{\int_0^1\int_0^\infty (\log u_0^2)\Phi r^3drdz}}$. This completes the proof of Theorem \ref{thm-bu-exterior}.

\section{Global Regularity with Large Data}

This section is devoted to proving Theorem \ref{thm1.2}.

\noindent
{\bf Proof of Theorem \ref{thm1.2}.}
First of all, the local well-posedness of the initial boundary value
problem
can be established by using an argument similar to that of \cite{HSW09}.
Now let us assume that $(u, \psi)$ is a local smooth solution
satisfying boundary condition \eqref{data-31} such that
$u \in C\big([0, T); H^s(\Omega(0, 1))\big)$,
$\psi \in C\big([0, T); H^{s+1}(\Omega(0, 1))\big)$ for some $T > 0$.
We are going to show that $T = \infty$.

Denote by
\begin{equation}\nonumber
\widetilde{v} = - \partial_z\psi,\quad \widetilde{u} = u^2.
\end{equation}
It is easy to see that
\begin{equation}\nonumber
\widetilde{v} \in C\big([0, T); H^s(\Omega(0, 1))\big),\quad
\widetilde{u} \in C\big([0, T); H^s(\Omega(0, 1))\big).
\end{equation}
Differentiating the second equation in \eqref{3DInviscidModel} with
$\nu=0$ with respect to $z$, we get
\begin{equation}\nonumber
(\partial_r^2 + \frac{3}{r}\partial_r +
\partial_z^2)\partial_t\widetilde{v} = \partial_z^2\widetilde{u}.
\end{equation}
Note that $\partial_t\widetilde{v}|_{\partial\Omega(0, 1)} = 0$.
We define $g \equiv \Delta_5^{-1} f$ as the solution of the
Laplacian equation with the homogeneous Dirichlet boundary condition:
\begin{equation}\nonumber
(\sum_{i = 1}^4\partial_{x_i}^2 +
\partial_z^2) g = f,\quad g|_{\partial\Omega(0,
1)} = 0,\quad \forall f \in L^2(\Omega(0, 1)).
\end{equation}
Then, we can reformulate the 3D inviscid model as follows:
\begin{align}\nonumber
\begin{cases}
\widetilde{u}_t =-4\widetilde{u}\widetilde{v},  \\
\widetilde{v}_t= \Delta_5^{-1}\partial_z^2\widetilde{u}.
\end{cases}
\end{align}
Note that the boundary condition in \eqref{data-31} implies that
\begin{equation}\nonumber
\widetilde{v}|_{\partial\Omega(0, 1)} = M .
\end{equation}
If we further denote
\begin{equation}\label{def-v}
v = \widetilde{v} - M,
\end{equation}
we obtain an equivalent system for $\widetilde{u}$ and $v$:
\begin{eqnarray}
\label{eqn-u}
&&\widetilde{u}_t = - 4M\widetilde{u} -4\widetilde{u}v ,  \\
\label{eqn-v}
&&v_t = \Delta_5^{-1}\partial_z^2\widetilde{u},\\
\label{BC-v}
&&v|_{\partial\Omega(0, 1)} = 0.
\end{eqnarray}

Recall the following well-known Sobolev inequality \cite{Foland95}.
Let $u,v \in H^s(\Omega)$ with $s>n/2$ ($n$ is the dimension of
$\Omega$). We have
\begin{eqnarray}
\label{Sobolev-1}
\|uv\|_{H^s(\Omega)}\le C_s \|u\|_{H^s(\Omega)}\|v\|_{H^s(\Omega)}.
\end{eqnarray}
We will also use the Poincar\'e inequality \cite{Evans98}:
\begin{eqnarray}
\label{Sobolev-2}
\|v\|_{H^s(\Omega)}\le C_s \|\nabla v\|_{H^{s-1}(\Omega)} \;,
\end{eqnarray}
if $v|_{\partial\Omega(0, 1)} = 0$. Here $C_s$ is an absolute positive
constant depending on $s$ and $\Omega$ only.

Taking the $H^s$ norm to the both sides of \eqref{eqn-u} and using
\eqref{Sobolev-1}-\eqref{Sobolev-2}, we obtain
\begin{eqnarray}
\label{eqn-4-1}
\frac{d}{dt} \|\widetilde{u}\|_{H^s} &\leq &  -4 M \|\widetilde{u}\|_{H^s}
+ 4C_s \|\widetilde{u}\|_{H^s} \|v \|_{H^s} \\ \nonumber
&\leq & -4 M \|\widetilde{u}\|_{H^s} + 4 C_s^2 \|\widetilde{u}\|_{H^s}
\|\nabla v \|_{H^{s-1}} \\ \nonumber
& \leq & (-4M + 4 C_s^2 \|\nabla v \|_{H^{s-1}}) \|\widetilde{u}\|_{H^s},
\end{eqnarray}
where we have used \eqref{BC-v}.

Next, we apply $\nabla$ to the both sides of \eqref{eqn-v} and take
the $H^{s-1}$-norm. We get
\begin{eqnarray}
\label{eqn-4-2}
\frac{d}{dt} \|\nabla v\|_{H^{s-1}} &\leq  &
\|\nabla \Delta_5^{-1}\partial_z^2 \widetilde{u} \|_{H^{s-1}} \\ \nonumber
&\leq &
\|\Delta_5^{-1}\partial_z^2 \widetilde{u} \|_{H^{s}} \\ \nonumber
&\leq & C_s\|\partial_z^2 \widetilde{u} \|_{H^{s-2}} \leq C_s \|\widetilde{u} \|_{H^{s}},
\end{eqnarray}
where we have used the standard elliptic regularity estimate for
$\Delta_5^{-1}$.

By the local well-posedness result and the assumption on the initial
condition \eqref{data-4}, we know that there exists a positive $T>0$ such
that we have
\begin{equation}
\label{eqn-4-3}
\|\nabla v (t)\|_{H^{s-1}} \leq \frac{M}{2 C_s^2} , \quad 0 \leq t \leq T.
\end{equation}
Let $T^*$ be the largest time such that \eqref{eqn-4-3} holds. We will
prove that $T^* = \infty$. Suppose $T^* < \infty$. Substituting
\eqref{eqn-4-3} into \eqref{eqn-4-1}, we obtain
\begin{equation}
\label{eqn-4-4}
\frac{d}{dt} \|\widetilde{u}\|_{H^s} \leq  -2 M \|\widetilde{u}\|_{H^s},
\quad 0 \leq t < T^*,
\end{equation}
which implies
\begin{equation}
\label{eqn-4-5}
\|\widetilde{u}\|_{H^s} \leq  \|\widetilde{u_0}\|_{H^s} e^{-2M t},
\quad 0 \leq t < T^*.
\end{equation}
Now substituting \eqref{eqn-4-5} into \eqref{eqn-4-2} yields
\begin{eqnarray}
\label{eqn-4-6}
\|\nabla v\|_{H^{s-1}} &\leq&  \|\nabla v_0\|_{H^{s-1}} +
C_s \|\widetilde{u_0}\|_{H^s} \int_0^t e^{-2Ms} ds \\ \nonumber
&\leq & \|\nabla v_0\|_{H^{s-1}} + \frac{C_s}{2M} \|\widetilde{u_0}\|_{H^s} \\ \nonumber
&\leq &\frac{M}{4 C_s^2} < \frac{M}{2 C_s^2}, \quad 0 \leq t < T^*,
\end{eqnarray}
where we have used the condition \eqref{data-4}. Since
$\|\nabla v\|_{H^{s-1}} \leq \frac{M}{4 C_s^2} < \frac{M}{2 C_s^2}$
for $t < T^*$ and
$\|\widetilde{u}\|_{H^s} \leq  \|\widetilde{u_0}\|_{H^s} e^{-2M T^*}
< \|\widetilde{u_0}\|_{H^s}$ for $t < T^*$, we can use
our {\it a priori} estimates \eqref{eqn-4-1}-\eqref{eqn-4-2} to
show that \eqref{eqn-4-3} remains valid for $ 0\leq t \leq T^* + \delta$
for some $\delta >0$. This contradicts with the assumption that $T^*$
is the largest time such that \eqref{eqn-4-3} holds. This contradiction
implies that $T^* = \infty$ and estimates \eqref{eqn-4-5}-\eqref{eqn-4-6}
remain valid for all times. This completes the proof of
Theorem \ref{thm1.2}.

\section{Local Well-Posedness in the Exterior Domain}

The local well-posedness theory of the initial-boundary value
problem of the 3D model \eqref{3DInviscidModel} with
\eqref{data-1} and \eqref{bd-1} is based on the following lemma
\ref{lemA-1}. Once the lemma is proved, the proof can be carried
out in exactly the same way as the local well-posedness analysis
presented in \cite{HSW09}. Hence, we will only present the
proof of the lemma \ref{lemA-1} in a slightly modified domain
with $1 < r < \infty$ and $0 < z < \pi$.
\begin{lem}\label{lemA-1}
For any given $\omega\in H^{s-2}(\Omega(1, \infty))$ with
$s\geq2$, there exists a unique solution $\psi\in H^s(\Omega(1,
\infty))$ to the boundary value problems
\begin{align}\label{A-2}
\begin{cases}
- \big(\partial_r^2 + \frac{3}{r}\partial_r
  + \partial_z^2\big)\psi = \omega,\quad 1 < r,\; 0 < z < \pi,\\
\psi_r +\beta\psi|_{r=1}=0,\quad \psi_z|_{z=0} = \psi_z|_{z=\pi} =
0,
\end{cases}
\end{align}
where $\beta \in S_{\rm exterior}$ is a constant. Furthermore, we
have the following estimate:
\begin{align}\label{A-est}
\|\psi\|_{H^s(\Omega(1, \infty))} \leq
C\|\omega\|_{H^{s-2}(\Omega(1, \infty))},
\end{align}
where $C$ is an absolute positive constant.
\end{lem}
\begin{proof}
Let us decompose
$\psi=\psi^{(1)}+\psi^{(2)},$
where $\psi^{(1)}$ is the solution of the elliptic equation
with the following mixed Dirichlet-Neumann boundary condition:
\begin{align}\label{A-3}
\begin{cases}
- \big(\partial_r^2 + \frac{3}{r}\partial_r
  + \partial_z^2\big)\psi^{(1)} = \omega,\quad {\rm in}\ \Omega(1, \infty),\\
\psi_z^{(1)}|_{z=0} = \psi_z^{(1)}|_{z=\pi}= 0, \quad \psi^{(1)}|_{r=1} = 0, \quad \psi^{(1)}
\rightarrow 0\ {\rm as}\ r \rightarrow \infty.
\end{cases}
\end{align}
Then $\psi^{(2)}$ satisfies
\begin{align}\label{A-4}
\begin{cases}
- \big(\partial_r^2 + \frac{3}{r}\partial_r
  + \partial_z^2\big)\psi^{(2)} = 0,\quad {\rm in}\ \Omega(1, \infty),\\
[\psi^{(2)}_r +\beta\psi^{(2)}]|_{r=1}= - \psi^{(1)}_r|_{r=1},\\
\psi^{(2)}_z|_{z=0}= \psi^{(2)}_z|_{z=\pi} = 0.
\end{cases}
\end{align}
The standard elliptic theory gives that $\psi^{(1)}\in
H^s(\Omega(1, \infty))$ and $\|\psi^{(1)}\|_{H^s(\Omega(1,
\infty))}\leq \|\omega\|_{H^{s-2}(\Omega(1, \infty))}$. It remains
to show that the Neumann-Robin problem \eqref{A-4} is
well-posed.

Let us perform the cosine transform to \eqref{A-4} with respect to
$z$ variable, which gives
\begin{align}\label{A-5}
\begin{cases}
- \big(\partial_r^2 + \frac{3}{r}\partial_r)\widehat{\psi^{(2)}} + k^2\widehat{\psi^{(2)}} = 0,\\
[\widehat{\psi_r^{(2)}}
+\beta\widehat{\psi^{(2)}}]|_{r=1}=-\beta\widehat{\psi^{(1)}_r}(1,k),
\end{cases}
\end{align}
where the cosine transform $\ \widehat{\cdot}\ $ is defined as:
\begin{align}
\widehat{\psi^{(2)}}(r,k)=\frac{2}{\pi}\int_0^{\pi}\psi^{(2)}(r,z)\cos(kz)dz.
\end{align}
Let $\widehat{\psi^{(2)}}(r,k)=\frac{1}{r}f(rk, k)$.
Then \eqref{A-5} can be written in terms of $f$ as
\begin{align}\nonumber
k^2\frac{d^2}{dr^2}f(rk)+\frac{k}{r}\frac{d}{dr}f(rk)-(k^2+\frac{1}{r^2})f(rk)=0,
\end{align}
which is equivalent to the following modified Bessel equation
\begin{align}\nonumber
r^2f''(r, k)+ rf'(r, k)-(1+r^2)f(r, k)=0.
\end{align}

The general solution of the modified Bessel equation with the
decaying property as $r \rightarrow \infty$ is given by modified
Bessel function:
\begin{align}\nonumber
f(r, k)= C(k)K(r),
\end{align}
where $K(r)$ is the modified Bessel function and can be
represented in an integration form as follows
\begin{align}\nonumber
K(r) = \int_0^{\infty}e^{-r\cosh(\theta)}\cosh(\theta)
d\theta,\quad K(r) \rightarrow 0\ {\rm as}\ r \rightarrow \infty.
\end{align}
In fact, we have $0 < K(r) \leq c_2r^{-1/2}e^{-r}$ as
$r \rightarrow \infty$ for some positive constant $c_2$. Thus, the
general solution of $\widehat{\psi^{(2)}}(k,r)$ is given by
\begin{align}
\widehat{\psi^{(2)}}(r,k)=\frac{C(k)}{r}K(rk).\label{A-10}
\end{align}
The boundary condition of $\psi^{(2)}$ in \eqref{A-5} can be used
to determine $C(k)$ as follows:
\begin{align}
C(k)=\frac{-\beta\widehat{\psi^{(1)}}(k,1)}{(\beta-1)K(k)+kK'(k)}.\label{A-11}
\end{align}
Note that $C(k)$ in \eqref{A-11} is well-defined for
$\beta\in S_{\rm exterior}$. The remaining part of the proof is
exactly the same as in \cite{HSW09}. We can show that
$\psi^{(2)}$ will have the same regularity property as $\psi^{(1)}$.
This proves our lemma. We will omit the details here.
\end{proof}

\section*{Acknowledgement}
This work was supported in part by NSF by Grant DMS-0908546. Zhen
Lei would like to thank the Applied and Computational Mathematics
of Caltech and Prof. Thomas Hou for hosting his visit and for
their hospitality during his visit. Zhen Lei was in part supported
by NSFC (grant No.11171072), the Foundation for Innovative
Research Groups of NSFC (grant No.11121101), FANEDD, Innovation
Program of Shanghai Municipal Education Commission (grant
No.12ZZ012) and SGST 09DZ2272900. The research of Dr. S. Wang was
supported by China 973 Program(Grant no. 2011CB808002), the Grants
NSFC 11071009 and PHR-IHLB 200906103.


\frenchspacing
\bibliographystyle{plain}







\end{document}